\documentclass[11pt,letterpaper]{amsart}
\usepackage{amsmath}
\usepackage{amstext}
\usepackage{amssymb}
\usepackage{amsfonts}
\usepackage{enumerate}
\usepackage{mathrsfs}
\usepackage{braket}
\usepackage{color}
\usepackage[all]{xy}
\usepackage{url}

\numberwithin{equation}{section}

\newtheoremstyle{note}
{1em}
{1em}
{}
{}
{\bfseries}
{:}
{.5em}
{}

\newtheorem{theorem}{Theorem}[section]
\newtheorem{lemma}[theorem]{Lemma}
\newtheorem{proposition}[theorem]{Proposition}
\newtheorem{corollary}[theorem]{Corollary}

\theoremstyle{note}
\newtheorem{remark}[theorem]{Remark}

\newtheorem{definition}[theorem]{Definition}

\newcommand{\Sph}{{\mathscr{S}}}
\newcommand{\Ball}{{\mathscr{B}}}

\newcommand{\N}{{\mathbb{N}}}
\newcommand{\R}{{\mathbb{R}}}
\newcommand{\C}{{\mathbb{C}}}

\newcommand{\n}[1]{ \left\|#1\right\| }

\DeclareMathOperator{\tr}{tr}

\DeclareMathOperator{\supp}{supp}

\DeclareMathOperator*{\argmin}{arg\,min}

\newcommand{\tn}[1]{{\left\vert\kern-0.25ex\left\vert\kern-0.25ex\left\vert #1 
    \right\vert\kern-0.25ex\right\vert\kern-0.25ex\right\vert}}

\title[Uniform homeomorphisms between spheres]{Uniform homeomorphisms between spheres of unitarily invariant ideals}

\author[J.A. Ch\'avez-Dom\'inguez]{Javier Alejandro Ch\'avez-Dom\'inguez}
\address{Department of Mathematics, University of Oklahoma, Norman, OK 73019-3103,
USA} \email{jachavezd@ou.edu}

\thanks{The author was partially supported by NSF grant DMS-1900985.}



\begin{document}

\maketitle

\begin{abstract}
The classical Mazur map is a uniform homeomorphism between the unit spheres of $L_p$ spaces, and the version for noncommutative $L_p$ spaces has the same property.

Odell and Schlumprecht used two types of generalized Mazur maps to prove that the unit sphere of a Banach space $X$ with an unconditional basis is uniformly homeomorphic to the unit sphere of a Hilbert space if and only if $X$ does not contain $\ell_\infty^n$'s uniformly. We prove a noncommutative version of this result, yielding uniform homeomorphisms between spheres of unitarily invariant ideals, and along the way we study noncommutative versions of the aforementioned generalized Mazur maps: one based on the $p$-convexification procedure, and one based on the minimization of quantum relative entropy.
The main result provides new examples of Banach spaces whose unit spheres are uniformly homeomorphic to the unit sphere of a Hilbert space (in fact, spaces with the  property (H) of Kasparov and Yu).
\end{abstract}

\section{Introduction}

The nonlinear theory of Banach spaces studies various notions of equivalence between Banach spaces.
On one end, if we consider only the topology everything collapses: any two separable infinite-dimensional Banach spaces are homeomorphic \cite{Kadets}.
On the other end, if we consider uniform homeomorphisms, things can collapse in the completely opposite way:
for $1<p<\infty$, if a Banach space $X$ is uniformly homeomorphic to $\ell_p$, then $X$ is linearly isomorphic to $\ell_p$ \cite{JLS}.
Considering uniform homeomorphisms between spheres gives a more flexible notion of equivalence between these two extremes, with the first example going back to Mazur \cite{Mazur}: for $p,q \in [1,\infty)$,  the map $f \mapsto \mathrm{sign}(f) |f|^{p/q}$ is a uniform homeomorphism between the spheres of $L_p(\mu)$ and $L_q(\mu)$.
It has been known for a while that the analogous result holds for noncommutative $L_p$ spaces (see e.g. \cite{Ozawa,Raynaud})  although precise estimates for the moduli of uniform continuity are rather recent \cite{Ricard}.
The main goal of this paper is to study other generalizations of the Mazur map, 
specifically noncommutative versions of two generalized Mazur maps used by Odell and Schlumprecht \cite{Odell-Schlumprecht} to characterize the unconditional sequence spaces whose unit spheres are uniformly homeomorphic to that of a Hilbert space: those which have finite cotype (that is, they do not contain a sequence of subspaces uniformly linearly isomorphic to the spaces $\ell_\infty^n$).

Let us now clarify what we mean by ``noncommutative''.
An oversimplified way of thinking about what changes mathematically when going from the commutative to the noncommutative points of view, is that vectors are replaced by matrices. Moreover, we can identify the classical vector case with the subspace of diagonal matrices. Since diagonal matrices commute but general ones do not, the term ``noncommutative'' is often used to describe the matrix version of a classical space.
The easiest example relates the $\ell_p$ norm on $\R^n$ or $\C^n$ with the Schatten $p$-norm for $n\times n$ matrices, as the latter evaluated on a diagonal matrix is precisely the former evaluated on the vector of its diagonal entries.  One can analogously define noncommutative versions of other sequence spaces as follows.
For a compact operator $A$ acting on $\ell_2$,  we denote by  $s(A) = \big(s_j(A)\big)_{j=1}^\infty$  the sequence of singular values of $A$ (i.e.  the eigenvalues of $|A|$).
If $E$ is a sequence space,  one can define a Schatten-like family of operators $S_E = \big\{ A : \ell_2 \to \ell_2 \text{ compact} : s(A) \in E \big\}$
and $\n{A}_{S_E} = \n{ s(A) }_E$; this will be the ``noncommutative'' version of the sequence space $E$. Under a technical condition on $E$ (namely  being $1$-symmetric, see Section \ref{sec-preliminaries} for more details), the associated family $S_E$ is an \emph{unitarily invariant ideal of operators} with norm given by $\|\cdot\|_{S_E}$,  where unitarily invariant of course means that $\|UAV\|_{S_E} = \|A\|_{S_E}$ whenever $U,V$ are unitaries on $\ell_2$.
These are the unitarily invariant ideals mentioned in the title, and we will be interested in uniform homeomorphisms between their spheres.

Let us point out that it is important, and typically not trivial, to understand which nice geometric properties of a sequence space $E$ are inherited by its noncommutative sibling $S_E$. While this is well understood for \emph{linear} properties such as uniform convexity/smoothness or cotype \cite{Arazy,Garling-TJ} and has been thoroughly studied since at least the early 1980's (see the survey \cite{Kaminska-Czerwinska}), there is a large gap in our knowledge of what happens for the \emph{nonlinear} geometric properties of Banach spaces that have been well-studied in the commutative setting for the past couple of decades.
The present paper is a small step in this wide open direction.

The first of the generalized Mazur maps used by Odell and Schlumprecht is given by the $p$-convexification procedure, which generalizes the fact that a sequence $(x_n)$  is in $\ell_p$ if and only if $(|x_n|^p) \in \ell_1$ and moreover  $\|(x_n)\|_p = \| (|x_n|^p) \|_1^{1/p}$.
Similarly, for a sequence space $E$ we can define its $p$-convexi\-fication $E^{(p)}$ as the set of sequences $(x_n)$ such that $(|x_n|^p) \in E$, with norm $\|(x_n)\|_{E^{(p)}} = \| (|x_n|^p) \|_E^{1/p}$.  The $p$-convexification generalized Mazur map $E^{(p)} \to E$ is then given by $(x_n) \mapsto \big( \mathrm{sign}(x_n)|x_n|^p\big) $.
Odell and Schlumprecht \cite[Prop. 2.8]{Odell-Schlumprecht} proved that for any $1\le p < \infty$ and any unconditional sequence space $E$ this map yields a uniform homeomorphism between the spheres of $E^{(p)}$ and $E$, and moreover the moduli of uniform continuity of this map and its inverse depend only on $p$.
A corresponding $p$-convexification map was already available for noncommutative sequence spaces, and we show in Section \ref{sec-convexification} that it is also a uniform homeomorphism between spheres by adapting the techniques of \cite{Ricard} (see Theorem \ref{thm-homeomorphism-convexification}). In contrast with the classical result of Odell and Schlumprecht, ours is restricted to the range $3 \le p < \infty$.

The second of the generalized Mazur maps of Odell and Schlumprecht is an entropy-minimization one. While they did not use this language, the entropy that they minimize is in fact (a translation of) the relative entropy from Information Theory.
This realization inspired us to use the notion of quantum relative entropy from Quantum Information Theory to define a noncommutative version of the entropy-minimization generalized Mazur map from \cite{Odell-Schlumprecht}. This is done in Section \ref{sec-entropy}, where we use it to obtain some other uniform homeomorphisms.

In Section \ref{sec-main} we prove the main result, a noncommutative version of the Odell--Schlumprecht theorem: if $E$ is a 1-symmetric sequence space, then the unit sphere of $S_E$ is uniformly homeomorphic to the unit sphere of a Hilbert space if and only if $S_E$ does not contain $\ell_\infty^n$'s uniformly.
It should be noted that we have not been able to fully prove this result by following the entropy minimization route, so we use complex interpolation arguments in the style of Daher \cite{Daher} instead.
We then note that the aforementioned spaces $S_E$ in fact satisfy a finer condition, namely the Property (H) of Kasparov and Yu.

\section{Notation and preliminaries}\label{sec-preliminaries}

Throughout this paper we consider only complex scalars.
We denote by $M_{k,n}$ the space of all $k$ by $n$ matrices with complex entries, 
and set $M_n = M_{n,n}$. 
We denote by $\tr(Z)$ the 
trace of a matrix $Z\in M_n$. 
We write $I_n$ (or simply~$I$) for the identity matrix in $M_n$.
The operator norm on $M_n$ will be denoted by $\n{\cdot}_\infty$,
and the trace norm by $\n{\cdot}_1$.
We denote by $M_n^+$ the subset of positive semidefinite matrices.

\subsection{Convexity and smoothness in Banach spaces}
The {unit sphere} of a Banach space $X$ will be denoted by $\Sph(X)$,
and its {closed unit ball} by $\Ball(X)$.
A Banach space $X$ is said to be \emph{strictly convex} if whenever $x,y\in\Sph(X)$ are distinct we have $\n{\tfrac{1}{2}(x+y)} < 1$.
The space $X$ is called \emph{uniformly convex} when the above condition holds uniformly, that is, for every $\varepsilon\in(0,2)$ the \emph{modulus of uniform convexity}
\[
\delta_X(\varepsilon) = \inf\left\{ 1-\n{\frac{x+y}{2}} \;:\; \|x\|=\|y\|=1, \|x-y\| \ge\varepsilon \right\}
\]
is strictly positive.
A Banach space $X$ is said to be \emph{smooth} if for every nonzero vector $x\in X$ there exists a unique {norming functional}, that is, $x^* \in\Sph(X^*)$ such that $x^*(x)=\n{x}$.
For a smooth Banach space $X$, ${J_X : X \to X^*}$ denotes the \emph{duality map}: $J_X(x)$ is the unique functional $x^* \in X^*$ such that $\n{x^*}=\n{x}$ and $x^*(x) = \n{x}^2$.
Sometimes we write only $J$ instead of $J_X$ if the space $X$ is clear from the context.
There is of course a notion of being \emph{uniformly smooth} with an accompanying modulus, but we will not need the precise details in this paper. We refer the reader to \cite[Appendix A]{Benyamini-Lindenstrauss} for a thorough account of the basics in this subject.

\subsection{Unconditional and symmetric bases}
A sequence $(x_n)_{n=1}^\infty$ in a Banach space $E$ is a \emph{(Schauder) basis} if every $x\in X$ admits a unique expansion as a convergent series $x=\sum_{n=1}^\infty a_nx_n$ for scalars $a_n$.
A Schauder basis is said to be \emph{unconditional} if the convergence of the series expansions is unconditional, that is, whenever $x=\sum_{n=1}^\infty a_nx_n$ and $(\lambda_n)$ is a bounded sequence of scalars, the series $\sum_{n=1}^\infty \lambda_na_nx_n$ converges. If additionally we have $\n{\sum_{n=1}^\infty \lambda_na_nx_n} \le \n{\sum_{n=1}^\infty a_nx_n}$ whenever $|\lambda_n| \le 1$, the basis is said to be \emph{1-unconditional}.
A Schauder basis is called \emph{symmetric} if for every $x=\sum_{n=1}^\infty a_nx_n$, all the series $\sum_{n=1}^\infty \lambda_na_nx_{\pi(n)}$ converge whenever $(\lambda_n)$ is a bounded sequence of scalars and $\pi : \N\to\N$ is a permutation.
If additionally we have $\n{\sum_{n=1}^\infty \lambda_na_nx_{\pi(n)}} \le \n{\sum_{n=1}^\infty a_nx_n}$ whenever $|\lambda_n| \le 1$, the basis is said to be \emph{1-symmetric}. Note that when the basis is 1-symmetric, the norm is permutation invariant in the sense that
$\n{\sum_{n=1}^\infty a_nx_{\pi(n)}} = \n{\sum_{n=1}^\infty a_nx_{n}}$.
The classical book \cite{Lindenstrauss-Tzafriri-I} is an excellent reference for the theory of Schauder bases.

Note that a space $E$ with a Schauder basis can be interpreted as a space of sequences, since every $x \in E$ corresponds to a (unique) sequence of coefficients in its series expansion with respect to the basis.
In this paper, we will identify a space with a Schauder basis with the corresponding space of sequences. Therefore, instead of talking about spaces with a 1-symmetric basis, we will talk about 1-symmetric sequence spaces (and similarly for the other notions).

As usual, when $1 \le p <\infty$ we denote by $E^{(p)}$ the $p$-convexification of a $1$-unconditional sequence space, that is, $E^{(p)}$ the space of sequences $a=(a_n)$ such that $|a|^p \in E$ (where the absolute value is understood coordinatewise) with norm $\n{a}_{E^{(p)}} = \n{|a|^p}_E^{1/p}$. The space $E^{(p)}$ is also a Banach $1$-unconditional sequence space, see \cite{Lindenstrauss-Tzafriri-II} for more on the $p$-convexification procedure on general Banach lattices.

\subsection{Unitarily invariant ideals}
Recall from the introduction that given a Banach sequence space $E$, $S_E$ is the set of compact operators $A : \ell_2 \to \ell_2$ whose sequence of singular values $s(A)$ belongs to $E$, and we denote  $\n{A}_{S_E} = \n{ s(A) }_E$.
If $E$ is a $1$-symmetric Banach sequence space, then $S_E$ is a unitarily invariant ideal of operators.
When $E=\ell_p$, we abbreviate $S_p := S_{\ell_p}$.
When we talk about unitarily invariant ideals, we mean ideals of the form $S_E$ as above; in particular we consider only ideals consisting of compact operators, so $\mathcal{B}(\ell_2)$ will not be included.
Standard references for unitarily invariant ideals are \cite{Gohberg-Krein,Simon}. These ideals could of course be defined over Hilbert spaces larger than $\ell_2$, but such extra generality would only complicate the notation so we stay in the separable case for simplicity.
We will, however, make heavy use of the finite-dimensional versions: $S_E^n$ denotes the space $M_n$ with the norm $\n{\cdot}_{S_E}$, and when $E = \ell_p$ we abbreviate $S_p^n := S_{\ell_p}^n$.
Since $S_E$ consists of bounded linear operators on $\ell_2$, we can interpret its elements as infinite matrices. By considering initial blocks, we have a standard embedding of $S_E^n$ into $S_E$. Many of our arguments in $S_E$ will be based on a reduction to the finite-dimensional case, and this is possible because of the following result (which should be well-known).

\begin{proposition}\label{prop-approximations-in-S_E}
Let $E$ be a $1$-symmetric sequence space. For any $A \in S_E$ we have $\lim_{m\to\infty}\n{A-P_mAP_m}_{S_E} = 0$ where $P_m : \ell_2 \to \ell_2$ is the orthogonal projection onto the span of the first $m$ elements of the canonical basis.
In particular, $\bigcup_{n=1}^\infty S_E^n$ is dense in $S_E$.
\end{proposition}

\begin{proof}
Let $A \in S_E$, and fix $\varepsilon>0$. Notice that the canonical approximations $(A_n)_{n=1}^\infty$ given by the singular value decomposition converge to $A$ in $S_E$ \cite[Thm. 2.7.(b)]{Simon}, so we can choose $n_0\in\N$ satisfying $\n{A-A_{n_0}}_{S_E} < \varepsilon$.
Now for any $m\in\N$,
\begin{multline*}
  \n{A-P_mAP_m}_{S_E} \le  \\
  \n{A-A_{n_0}}
_{S_E} + \n{A_{n_0} - P_mA_{n_0}P_m}_{S_E} + \n{P_m(A_{n_0} - A)P_m}_{S_E}.  
\end{multline*}
The first term is smaller than $\varepsilon$
 by assumption, and so is the third by the ideal property. We just need to make sure that the middle term is less than $\varepsilon$ for $m$ large enough, which is clear because $A_{n_0}$ is a fixed finite-rank operator.
 \end{proof}

The following result about complex interpolation of unitarily invariant ideals is folklore: for example, it is used without comment in the proof of \cite[Prop. 2]{TJ-uc}. We point out that one can prove it in a straightforward manner by using \cite[Thm. 2.10]{Simon}.

\begin{proposition}\label{prop:interpolation-unitarily-invariant-ideals}
Let $E_0,E_1$ be an interpolation pair of 1-symmetric sequence spaces. Then for any $\theta \in[0,1]$ we have $(S_{E_0},S_{E_1})_\theta = S_{(E_0,E_1)_\theta}$ isometrically.
\end{proposition}

\section{Unitarily invariant matrix norms}

By a unitarily invariant norm on $M_n$ we mean a norm $\tn{\cdot}$ such that for any $A,U,V \in M_n$ with $U,V$ unitaries we have $\tn{A} = \tn{UAV}$.
Sometimes we call these \emph{unitarily invariant matrix norms} to emphasize the distinction between them and the them the unitarily invariant norms on infinite-dimensional spaces from the previous section.

Unitarily invariant matrix norms are a well-studied subject, we refer the reader to e.g. \cite[Sec. IV.2]{Bhatia} for details. In the rest of this section we will record some properties that we will need in this paper.

\begin{proposition}\label{prop-unitarily-invariant-matrix-norm-basics}
\begin{enumerate}[(a)]
\item A norm $\tn{\cdot}$ on $M_n$ is unitarily invariant if and only if it satisfies the ideal property, that is, for any $A,B,C\in M_n$ we have $\tn{ABC} \le \n{A}_\infty \, \tn{B} \, \n{C}_\infty$ \cite[Prop. IV.2.4]{Bhatia}.
\item  Unitarily invariant norms on $M_n$ are in bijective correspondance with $1$-symmetric norms on $\R^n$ (via the finite-dimensional version of the procedure $E \mapsto S_E$ from the previous section) \cite[Thm. IV.2.1]{Bhatia}.
\item The dual of a unitarily invariant norm with respect to trace duality is also unitarily invariant, and their associated sequence spaces are dual to each other \cite[Prop. IV.2.11]{Bhatia}.
\item A unitarily invariant norm on $M_n$ is smooth if and only it its associated sequence space is smooth \cite{Arazy}.
\end{enumerate}
\end{proposition}

We will need lemmas related to the last statement in the previous proposition, which give extra information about the duality map. These are not difficult and are surely known to experts, but we have not found them explicitly stated in the literature so we prove them for completeness.

\begin{lemma}\label{lemma-G-gives-states}
If $X=(M_n, \n{\cdot})$ is an unitarily invariant smooth norm and $\sigma \in \Sph(X)^+$, then $J_X(\sigma) \in \Sph(X^*)^+$ and $\sigma J_X(\sigma) = J_X(\sigma) \sigma \in \Sph(S_1^n)^+$.
\end{lemma}

\begin{proof}
We will assume that $\sigma$ is invertible, the general case then follows by approximation.
Since $\sigma \ge 0$, there exists an orthonormal basis of $\C^n$ so that the matrix representation of $\sigma$ with respect to this basis is the diagonal matrix corresponding to the vector $s \in \R^n$ of the singular values of $\sigma$. Note that by our assumption, all entries of $s$ are strictly positive. Moreover, if $E$ is the sequence space associated to $X$, $\n{\sigma}_X = \n{s}_E = 1$.
Note that the vector $J_E(s) \in E^*$ must have nonnegative entries: if an entry were negative, replacing it by its absolute value would give a functional  $y^* \in \Sph(E^*)$ satisfying $y^*(s) > \big(J_E(s)\big)(s)$, a contradiction.
If $\eta$ is the diagonal matrix (with respect to the same orthonormal basis as $\sigma$) corresponding to the vector $J_E(s) \in E^*$, we then clearly have $\eta \ge 0$, $\eta \sigma = \sigma\eta \ge 0$, $\tr( \eta \sigma  ) = 1$ and 
from Proposition \ref{prop-unitarily-invariant-matrix-norm-basics}.(c) $\n{\eta}_{X^*} = \n{J_E(s)}_{E^*} = 1$.
Therefore, since $X$ is smooth, $\eta = J_X(\sigma)$.
The only part of the conclusion we have not explicitly checked is that $\eta\sigma \in \Sph(S_1^n)$, but this is clear since $\eta\sigma \ge 0$ and $\tr(\eta\sigma) = 1$.
\end{proof}

\begin{lemma}\label{lemma-unique-factorization}
Let $X=(M_n, \n{\cdot})$ be an unitarily invariant smooth norm.
If $\sigma_1,\sigma_2 \in \Sph(X)^+$, satisfy $J_X(\sigma_1)\sigma_1 = J_X(\sigma_2)\sigma_2$, then $\sigma_1 = \sigma_2$.
\end{lemma}

\begin{proof}
Note that the corresponding statement for $1$-unconditional norms on finite-dimensional spaces is true: this is contained in the proof of \cite[Prop. 2.6]{Odell-Schlumprecht}, see also \cite[Lemma 9.5.(ii)]{Benyamini-Lindenstrauss}.
The version for the unitarily invariant norm $X$ then follows by relating the duality map of $X$ with the duality map for the associated $1$-symmetric norm, just as we did in the proof of Lemma \ref{lemma-G-gives-states}.
\end{proof}

When working with $L_p$ spaces, both commutative and noncommutative, two fundamental tools are H\"older's inequality and complex interpolation. In particular, these two tools are heavily used in the arguments of \cite{Ricard}. To adapt said arguments to the setting of the present paper, we will need versions of the aforementioned two tools in the setting of unitarily invariant matrix norms. The first one is quite well-known \cite[Ex. IV.2.7]{Bhatia}.

\begin{proposition}[H\"older's inequality for unitarily invariant matrix norms]
If $\tn{\cdot}$ is a unitarily invariant norm on $M_n$, and $p,q,r \ge 1$ satisfy
 $1/p+1/q=1/r$, then for any $A,B \in M_n$ we have 
\[
\tn{|AB|^r}^{1/r} \le \tn{ |A|^p }^{1/p} \, \tn{|B|^q}^{1/q}.
\]
\end{proposition}
See \cite[Thm. 2.8]{Simon} for a far reaching generalization, showing that any H\"older-type inequality for sequence spaces implies an analogous one for the corresponding unitarily invariant ideals.

The second result we will use, which will play the role of interpolation, is the following lemma.
It does not seem to be as well-known as H\"older's inequality (though it does appear in \cite[Ex. 2.7.12.(i)]{PDM}), so we provide a proof for completeness.

\begin{lemma}\label{lemma-contractions}
Let $T : M_n \to M_n$ be a linear operator. If $T$ is a contraction with respect to both the operator and trace norms, then it is a contraction with respect to any unitarily invariant norm on $M_n$.
\end{lemma}

\begin{proof}
By the Fan Dominance Theorem \cite[Thm. IV.2.2]{Bhatia}, it suffices to check that $T$ is a contraction with respect to the Ky Fan norms $\n{\cdot}_{(k)}$, $k=1,\dotsc,n$. Note that from hypothesis we already have this for $k=1,n$, as $\n{\cdot}_{(1)}$ is the operator norm and $\n{\cdot}_{(n)}$ is the trace norm. Fix $1<k<n$.
Fix $A \in  M_n$, and write it as $A = B + C$ with $B,C \in M_n$.
By \cite[Prop. IV.2.3]{Bhatia}, since $T(A) = T(B) + T(C)$,
$$
\n{T(A)}_{(k)} \le \n{T(B)}_{(n)} + k \n{T(C)}_{(1)} \le  \n{B}_{(n)} + k \n{C}_{(1)}.
$$
Taking the infimum over all such $B,C$ and using \cite[Prop. IV.2.3]{Bhatia} again, we conclude $\n{T(A)}_{(k)} \le \n{A}_{(k)}$.
\end{proof}

\section{The $p$-convexification generalized Mazur map}\label{sec-convexification}

The main goal of this section is to prove a noncommutative version of \cite[Prop. 2.8]{Odell-Schlumprecht}.
Our proofs follow closely the spirit of \cite{Ricard},
where an analogous result was proved for the case of noncommutative $L_p$ spaces.

Recall that for a $1$-unconditional sequence space its $p$-convexification is defined by $\n{x}_{E^{(p)}} = \n{|x|^p}_E^{1/p}$.
Therefore, it follows easily that when $E$ is $1$-symmetric we have
$$
\n{A}_{S_{E^{(p)}}} = \n{|A|^p}_{S_E}^{1/p}.
$$
For $1\le p < \infty$ define $G_p :  S_{E^{(p)}} \to S_E$ by $G_p(x) = u|x|^p$, where $x$ has polar representation $x=u|x|$.
Note that $G_p$ maps $\Sph(S_{E^{(p)}})$ to $\Sph(S_E)$.
The remainder of this section will be devoted to proving that $G_p$ is a uniform homeomorphism, with the moduli of continuity of $G_p$ and $G_p^{-1}$ depending only on $p$.

\begin{remark}\label{remark-2-by-2-self-adjoint}
We will repeatedly use a ``$2\times 2$ trick'' that allows us to reduce inequalities to the case of self-adjoint elements.
Recall that elements of $S_E$ can be understood as (infinite) matrices.
If $x \in S_E$, define the self adjoint element
$$
\widetilde{x} = \begin{pmatrix}
 0&x\\
 x^*&0
 \end{pmatrix}\
$$
and observe that
$$
2\n{x}_{S_E} \ge \n{\widetilde{x}}_{S_E} \ge \n{x}_{S_E}.
$$
\end{remark}

The following Lemma corresponds to \cite[Lemma 2.1]{Ricard}.

\begin{lemma}\label{lemma-Gp-inverse-uc-on-positive-elements}
Let $1\le p < \infty$, and let $x, y \in S_{E^{(p)}}^+$.
Then
$$
\n{x-y}^p_{S_{E^{(p)}}} \le \n{x^p - y^p}_{S_E}.
$$
\end{lemma}

\begin{proof}
By density, we may assume $x$ and $y$ are finitely supported. 
Note that
$$
\n{|x-y|^p}_{S_E} \le \n{ x^p-y^p }_{S_E},
$$
since this inequality holds for any unitarily invariant norm \cite[Eq. X.10]{Bhatia}, which is the desired result.
\end{proof}

Note that the previous Lemma gives the uniform continuity of $G_p^{-1}$ on positive elements.

Our next Lemma corresponds to \cite[Lemma 2.2]{Ricard}.

\begin{lemma}\label{lemma-Gp-uc-on-positive-special-case}
Let $\theta \in (0,1]$ and let $x, y \in S_{E^{(1+\theta)}}^+$.
Then
$$
\n{x^{1+\theta}-y^{1+\theta}}_{S_E} \le 3 \n{x-y}_{S_{E^{(1+\theta)}}} \max\big\{ \n{x}_{S_{E^{(1+\theta)}}}, \n{y}_{S_{E^{(1+\theta)}}} \big\}^\theta.
$$
\end{lemma}

\begin{proof}
Once again, we will assume that $x$ and $y$ are finitely supported.
Note that in this case, the conjugate exponent to $p=1+\theta$ is $p'=(1+\theta)/\theta$.
Without loss of generality, by density, we may assume that $x$ and $y$ are invertible. The restriction on $\theta$ allows us to have the integral representation
$$
s^{1+\theta} = c_\theta \int_{\R_+} \frac{t^\theta s^2}{s+t} \,
\frac{dt}{t}.
$$ 
On invertible elements the maps $f_t:s\mapsto
\frac{s^2}{s+t}= s (s+t)^{-1}s$ are differentiable and 
$$
D_sf_t(\delta)= \delta (s+t)^{-1}s+s(s+t)^{-1}\delta-s(s+t)^{-1}\delta (s+t)^{-1}s.
$$
Hence putting $\delta=x-y$, we get the integral representation
 $$
 x^{1+\theta}-y^{1+\theta}= c_\theta \int_0^1 \int_{\R_+} t^\theta 
D_{y+u\delta} f_t(\delta) \,\frac{dt}{t} {du}.
$$
We get, letting $g_t(s)=s(s+t)^{-1}$
\begin{multline*}
    x^{1+\theta}-y^{1+\theta} \\
=    \int_0^1 \Big((y+u\delta)^\theta \delta +\delta 
(y+u\delta)^\theta\Big) \,du-
c_\theta \int_0^1 \int_{\R_+} t^\theta g_t(y+u\delta) \delta g_t(y+u\delta)\,\frac{dt}{t} du.
\end{multline*}
To estimate the first term, we use the H\"older inequality for unitarily invariant norms. Note that we are also using the fact that $y+u\delta = (1-u)y + ux$ is positive.
\begin{multline*}
\n{ (y+u\delta)^\theta \delta }_{S_E} \le \n{(y+u\delta)^\theta}_{S_{E^{(p')}}}   \n{\delta}_{S_{E^{(p)}}}
 = \n{(y+u\delta)^{\theta p'}}_{S_{E}}^{1/p'} \n{x-y}_{S_{E^{(p)}}}  \\
 = \n{y+u\delta}_{S_{E^{(p)}}}^\theta \n{x-y}_{S_{E^{(p)}}} = \n{(1-u)y+ux}_{S_{E^{(p)}}}^\theta \n{x-y}_{S_{E^{(p)}}} 
 \\
 \le\max\big\{ \n{x}_{S_{E^{(1+\theta)}}}, \n{y}_{S_{E^{(1+\theta)}}} \big\}^\theta \n{x-y}_{S_{E^{(p)}}}.
\end{multline*}
The term $\delta 
(y+u\delta)^\theta$ can be estimated analogously.

When $u$ is fixed, note that $g_t(y+u\delta)$ is an invertible positive contraction. Put 
$$\gamma^2=c_\theta\int_{\R_+} t^\theta g_t(y+u\delta)^2  \frac{dt}{t}\leq (y+u\delta)^\theta,$$
and write $g_t(y+u\delta)=v_t\gamma$ so that 
$v_t$ and $y+u\delta$ commute and 
$$c_\theta\int_{\R_+} t^\theta v_t^2 \frac{dt}{t}=1.$$ Therefore the
 map defined on $M_n$, $z\mapsto c_\theta\int_{\R_+} t^\theta
 v_t z v_t \frac{dt}{t}$ is unital, completely
 positive and trace preserving, hence it is a contraction with respect to both the operator and trace norms, so by Lemma \ref{lemma-contractions} it is a contraction on $S_E$.
Applying it to $z= \gamma \delta \gamma$,
we deduce using the H\"older inequality again
\begin{multline*}
\Big\|c_\theta \int_{\R_+} t^\theta g_t(y+u\delta) \delta g_t(y+u\delta)\,\frac{dt}{t}\Big\|_{S_E} \leq \big\| \gamma \delta \gamma\big\|_{S_E} \\
\leq \big\|\delta\big
\|_{S_{E^{(1+\theta)}}} \big\|\gamma\big\|_{S_{E^{(2\theta/(1+\theta))}}}^2 
\leq  \big\|\delta\big\|_{S_{E^{(1+\theta)}}}  \big\|(y+u\delta)^{\theta/2}\big\|_{S_{E^{(2(1+\theta)/\theta)}}}^2 \\
= \big\|\delta\big\|_{S_{E^{(1+\theta)}}}  \big\|y+u\delta\big\|_{S_{E^{(1+\theta)}}}^\theta
\le \max\big\{ \n{x}_{S_{E^{(1+\theta)}}}, \n{y}_{S_{E^{(1+\theta)}}} \big\}^\theta \n{x-y}_{S_{E^{(1+\theta)}}}.
\end{multline*}
Putting all the estimates together yields the desired conclusion.
\end{proof}

The next result corresponds to \cite[Cor. 2.3]{Ricard}

\begin{corollary}\label{cor-Gp-uc-on-positive}
Let $1\le p < \infty$ and let $x, y \in S_{E^{(p)}}^+$.
Then
$$
\n{x^{p}-y^{p}}_{S_E} \le 3p \n{x-y}_{S_{E^{(p)}}} \max\big\{ \n{x}_{S_{E^{(p)}}}, \n{y}_{S_{E^{(p)}}} \big\}^{p-1}.
$$
\end{corollary}

\begin{proof}
When $p=n \in \N$, the result is clear even with constant $p$ instead of $3p$.
For the general case, let $m = \lfloor p \rfloor$, so that $p=m(1+\theta)$ for some $\theta \in (0,1]$.
Then apply the case of $m$ and Lemma \ref{lemma-Gp-uc-on-positive-special-case}, using the fact that convexification behaves well with respect to reiteration.
\end{proof}

Observe that the previous corollary gives the uniform continuity of $G_p$ on positive elements.
To extend this to general elements, we will again use $2\times 2$ tricks as in \cite{Ricard}.
First, let us observe that we can reduce to the case of self-adjoint elements.
Let $x, y\in\Sph(S_{E^{(p)}})$ 
with polar decompositions 
$x=u|x|$ and $y=v|y|$.
Now define
 $$\widetilde x=\begin{pmatrix}
 0&x\\
 x^*&0
 \end{pmatrix}\quad\text{and}\quad \widetilde y=\begin{pmatrix}
 0&y\\
 y^*&0
 \end{pmatrix}\,.$$
They are selfadjoint with  polar decompositions
$$\widetilde x=\widetilde u |\widetilde x|=\begin{pmatrix}
 0&u\\
 u^*&0
 \end{pmatrix} \cdot \begin{pmatrix}
u|x|u^*& 0\\
 0 & |x|
 \end{pmatrix}\quad\textrm{and}\quad 
\widetilde y=\widetilde v |\widetilde y|=\begin{pmatrix}
 0&v\\ v^*&0
 \end{pmatrix}\cdot \begin{pmatrix}
v|y|v^*& 0\\
 0 & |y|
 \end{pmatrix}.$$
Therefore
$$\widetilde u |\widetilde x|^{p}=\begin{pmatrix}
 0&u |x|^{p}\\
 |x|^{p}u^*&0
 \end{pmatrix}\quad\text{and}\quad \widetilde v |\widetilde y|^{p}=\begin{pmatrix}
 0&v |y|^{p}\\
 |y|^{p}v^*&0
  \end{pmatrix},$$ 
  we then have
  $$
 \widetilde u |\widetilde x|^{p} - \widetilde v |\widetilde y|^{p}=  \begin{pmatrix}
0& u|x|^p-v|y|^p\\
(u|x|^p-v|y|^p)^* & 0
 \end{pmatrix},
  $$
and thus, as in Remark \ref{remark-2-by-2-self-adjoint}, the norms of $x-y$ and $\widetilde{x}-\widetilde{y}$ are comparable, and so are the norms of $G_p(x)-G_p(y)$ and $G_p(\widetilde{x})-G_p(\widetilde{y})$.

 Next, we reduce the desired result to a commutator estimate by using the $2\times2$-trick again.
 We use the commutator notation $[x,b]=xb-bx$.
 Put 
 $$\widetilde x=\begin{pmatrix}
  x&0\\
  0&y
  \end{pmatrix}\quad\text{and}\quad \widetilde b=\begin{pmatrix}
  0&1\\
  0&0
  \end{pmatrix}\,.$$
 So that
 $$ 
 \big\| [G_p(\widetilde{x}) ,\widetilde b] \big\|_{S_E}= \big\| G_p(x)-G_p(y)\big\|_{S_E} \qquad\textrm{and}\qquad \big\| [\widetilde x,\widetilde b] \big\|_{S_{E^{(p)}}}=\big\| x-y\big\|_{S_{E^{(p)}}}.
 $$

\begin{lemma}\label{lemma-commutator}
Let $1\le p < \infty$,  $x\in S_{E^{(p)}}^+$ and $b\in S_\infty$ then
$$
\big\| [x, b] \big\|_{S_{E^{(p)}}} \le 4\cdot 2^{1/p}\,  \Big\| \big[x^{p},b\big]\Big\|_{S_E}^{1/p},
$$
$$\Big\| \big[x^{p},b\big]\Big\|_{S_E} \leq 4\cdot3p\cdot 2\, \n{x}^{p-1}_{S_{E^{(p)}}} \big\| [x, b] \big\|_{S_{E^{(p)}}} .$$

\end{lemma}

\begin{proof}
As usual, we may assume that $x$ and $b$ are finitely supported.
We may assume $\|b\|_\infty=1$ by homogeneity. Using the $2\times 2$-trick with 
$$\widetilde x=\begin{pmatrix}
 x& 0\\
 0& x
 \end{pmatrix}\quad\textrm{and}\quad \widetilde b=\begin{pmatrix}
 0&b\\
 b^*&0
 \end{pmatrix}\,,$$
we may assume $b=b^*$, while losing at most a factor of 4 in the constant.

Next, as $b=b^*$, we may use the Cayley transform defined by
$$u=(b-i)(b+i)^{-1} ,\qquad b=2i (1-u)^{-1}-i.$$
Clearly $u$ is unitary and functional calculus gives that $\|(1-u)^{-1}\|_\infty\leq \frac 1 {\sqrt 2}$. 
For the first inequality we have, using Lemma \ref{lemma-Gp-inverse-uc-on-positive-elements},
\begin{eqnarray*}
\big\| [x, b] \big\|_{S_{E^{(p)}}} &\leq& 2 
\big\| x (1-u)^{-1} -(1-u)^{-1} x \big\|_{S_{E^{(p)}}}\\
& \leq & 2 \big\|(1-u)^{-1}\big\|_\infty^2 \big\| x (1-u) -(1-u) x \big\|_{S_{E^{(p)}}}\\ & \leq& \big\| u^*x u - x \big\|_{S_{E^{(p)}}}\\
& \leq &  \big\| u^*x^p u - x^p\big\|_{S_E}^{1/p}  \\
& = &  \big\| x^p u - ux^p\big\|_{S_E}^{1/p}  \\
&\leq&  \big\|(b+i)^{-1}\big\|_\infty^{2/p}  \big\| (b+i)x^p(b-i)  - (b-i)x^p(b+i)\big\|_{S_E}^{1/p}\\
&\leq &  2^{1/p}\,  \big\| x^p b -b x^p \big\|_{S_E}^{1/p}.
\end{eqnarray*} 

Similarly, for the second inequality, but now using Corollary \ref{cor-Gp-uc-on-positive},
\begin{eqnarray*}
\big\| [x^{p}, b] \big\|_{S_E} &\leq& 2 
\big\| x^{p} (1-u)^{-1} -(1-u)^{-1} x^{p} \big\|_{S_E}\\
& \leq & 2 \big\|(1-u)^{-1}\big\|_\infty^2 \big\| x^{p} (1-u) -(1-u) x^{p} \big\|_{S_E}\\ & \leq& \big\| u^*x^{p} u - x^{p} \big\|_{S_E}\\
& \leq & 3p \big\| x u - ux\big\|_{S_{E^{(p)}}} \n{x}^{p-1}_{S_{E^{(p)}}} \\
&\leq& 3p \n{x}^{p-1}_{S_{E^{(p)}}}  \big\|(b+i)^{-1}\big\|_\infty^{2}  \big\| (b+i)x(b-i)  - (b-i)x(b+i)\big\|_{S_{E^{(p)}}}\\
&\leq & 3p\cdot 2\, \n{x}^{p-1}_{S_{E^{(p)}}}  \big\| x b -b x \big\|_{S_{E^{(p)}}}.
\end{eqnarray*} 
\end{proof}

Our next Lemma corresponds to \cite[Lemma 2.5]{Ricard}

\begin{lemma}\label{lemma-sum}
Let  $p \ge 1$.
There exists a constant $C$ such that for any
$x, y\in S_{E^{(p)}}^+$ and $b\in S_\infty$ we have
 $$
 \big\| x^pb + by^p \big\|_{S_E} \le C \n{x}_{S_{E^{(p)}}}^{p-1} \Big\| xb + by\Big\|_{S_{E^{(p)}}}
 $$
If $p \ge 3$, then there exists a constant $C_p$ such that 
for any $x, y\in S_{E^{(p)}}^+$ and $b\in S_\infty$ we have
$$\Big\| xb + by\Big\|_{S_{E^{(p)}}} \leq C_p
 \big\|b\big\|_\infty ^{1-1/p}\big\| x^pb + by^p \big\|^{1/p}_{S_E} .$$

\end{lemma}

\begin{proof}
Once more we assume that $x,y$ and $b$ are finitely supported.
Using the $2\times 2$ trick we may assume $x=y$.
We will be using \cite[Cor. IX.4.10]{Bhatia}:
for any $A,B \in M_n^+$ and $X \in M_n$, and any unitarily invariant norm $\n{\cdot}$, and $\alpha \in [0,1]$,
\begin{equation}\label{eqn-schur}
\Big\| A^{1-\alpha} X B^\alpha + A^\alpha X
B^{1-\alpha}\Big\| \leq  \Big\| AX+XB\Big\|.
\end{equation}
Now, observe that
$$
x^pb+bx^p = x^{p-1}(xb+bx)+(xb+bx)x^{p-1}-(x^{p-1}bx+xbx^{p-1}).
$$
When $p\le 3$,
by the H\"older inequality we get
$$
 \big\| x^pb + bx^p \big\|_{S_E} \le \n{x}_{S_{E^{(p)}}}^{p-1} \Big( 2\Big\| xb + bx\Big\|_{S_{E^{(p)}}} + \Big\| x^{\frac{p-1}{2}}bx^{\frac{3-p}{2}} + x^{\frac{3-p}{2}} bx^{\frac{p-1}{2}}\Big\|_{S_{E^{(p)}}} \Big)
$$
and using \eqref{eqn-schur} with $\alpha = \frac{p-1}{2}$ this yields
$$
 \big\| x^pb + bx^p \big\|_{S_E} \le \n{x}_{S_{E^{(p)}}}^{p-1} 3\Big\| xb + bx\Big\|_{S_{E^{(p)}}}.
$$
When $p>3$, another application of H\"older's inequality gives
$$
\big\| x^{p-1}bx+xbx^{p-1} \big\|_{S_E}
\le 2 \n{x}_{S_{E^{(p)}}}^{p-1} \n{ x^{1/2}bx^{1/2} }_{S_{E^{(p)}}}
$$
and using \eqref{eqn-schur} with $\alpha = 1/2$ we conclude
$$
\big\| x^{p-1}bx+xbx^{p-1} \big\|_{S_E}
\le \n{x}_{S_{E^{(p)}}}^{p-1} \n{xb+bx }_{S_{E^{(p)}}},
$$
and thus we have obtained the first inequality.

The second inequality is an immediate consequence of  \cite[Thm. 3.1]{Jocic-norm-ineqs}, which gives $C_p = 2^{1-1/p}$.
\end{proof}

The next Lemma corresponds to \cite[Lemma 2.6]{Ricard}.

\begin{lemma}\label{lemma-commutator-self-adjoint} 
There is an absolute constant $C>0$ and constants $C_p$ ($p>1$) so that for
 $x\in S_{E^{(p)}}$ with $x=x^*$ and $b\in S_\infty$ we have
\begin{equation}\label{eqn-commutator-self-adjoint-1}
\Big\| \big[G_{p}(x),b\big] \Big\|_{S_E} \leq C \n{x}_{S_{E^{(p)}}}^{p-1} \Big\| [x,b] \Big\|_{S_{E^{(p)}}},
\end{equation}
\begin{equation}\label{eqn-commutator-self-adjoint-2}
\Big\|  [x,b] \Big\|_{S_{E^{(p)}}} \leq C_p
\Big\| \big[G_{p}(x),b\big] \Big\|^{1/p}_{S_E}.
\end{equation}
\end{lemma}

\begin{proof}
Once more we may assume that $x$ and $b$ are finitely supported.
For \eqref{eqn-commutator-self-adjoint-1}, write $e_+=1_{[0,\infty)}(x)$ and $e_-=1_{(-\infty,0)}(x)$ and put 
$b_{\pm,\pm}=e_{\pm} be_{\pm}$. Then
$$\big[G_{p} (x),b\big]= \big[x_+^{p},b_{+,+}\big] - \,\big[x_-^{p},b_{-,-}\big]+ \,
\big( x_+^{p}b_{+,-}+b_{+,-}x_-^{p}\big) - \big( x_-^{p}b_{-,+}+b_{-,+}x_+^{p}\big).$$
We can apply either Lemma \ref{lemma-commutator} or \ref{lemma-sum} to each term. 
In any case, because of the ideal property of unitarily invariant norms the upper bound we get is smaller than the right side of \eqref{eqn-commutator-self-adjoint-1}
since
\begin{multline*}
\big[x_+,b_{+,+}\big] = e_+[x,b]e_+, \quad
\big[x_-,b_{-,-}\big] = e_-[x,b]e_-,\\
x_+b_{+,-}+b_{+,-}x_-= e_+( xb+bx )e_-, \quad
x_-b_{-,+}+b_{-,+}x_+= e_-( xb+bx )e_+.
\end{multline*}

A similar argument works for \eqref{eqn-commutator-self-adjoint-2}.
\end{proof}

From the Lemma above and the earlier arguments, we have proved the main result of this section.

\begin{theorem}\label{thm-homeomorphism-convexification}
Let $3 \le p<\infty$, and let $E$ be a 1-symmetric sequence space.
Then the map $G_p$ is a uniform homeomorphism between $\Sph(S_{E^{(p)}})$ and $\Sph(S_E)$.
Moreover, the moduli of continuity of $G_p$ and $G_p^{-1}$ depend only on $p$.
\end{theorem}

\begin{remark}
We do not know whether the restriction $p \ge 3$ in  Theorem \ref{thm-homeomorphism-convexification} is actually needed, and we suspect it is not.
As pointed out to the author by D. Carando, as $p \to 1$ the map $G_p$ becomes the identity so intuitively it should be easier for $G_p$ to be a uniform homeomorphism when $p$ is close to $1$.
Note also that the restriction was inherited from \cite[Thm. 3.1]{Jocic-norm-ineqs}, and if the case $1 < p<3$ were to hold for \cite[Thm. 3.1]{Jocic-norm-ineqs} then it would automatically also hold in Theorem \ref{thm-homeomorphism-convexification}.
\end{remark}

\section{The entropy minimization generalized Mazur map}\label{sec-entropy}

Though not explicitly stated in this way, the second generalized Mazur map from \cite{Odell-Schlumprecht} can be defined using the notion of relative entropy from Classical Information Theory. To be precise: given a fixed probability vector $x=(x_j)_{j=1}^n$, they minimize the quantity $-\sum_{j=1}^n x_j \log y_j$ as $y=(y_j)_{j=1}^n$ varies over the positive part of the unit ball of an unconditional norm on $\R^n$.
Since $x$ is fixed, minimizing the above quantity is equivalent to minimizing the relative entropy or Kullback--Leibler divergence $D(x||y) = \sum_{j=1}^n x_j(\log(x_j) - \log(y_j))$.
By using the corresponding notion of quantum relative entropy from Quantum Information Theory, we will define a quantum version of the second generalized Mazur map from \cite{Odell-Schlumprecht}.
Some of our arguments (e.g. Lemma \ref{lemma-uniform-continuity}, Propositions \ref{prop-FX-unif-coninuous on positive} and \ref{prop-G-unif-coninuous}) follow closely those from \cite{Odell-Schlumprecht}, but others (e.g. Lemma \ref{lemma-G-is-F-inverse}) are more specific to quantum entropy. 

Let us start by recalling the definition of quantum relative entropy. 

\begin{definition}
If $\rho \in M_n$ is a state (that is, $\rho \ge 0$ and $\tr(\rho)=1$), and $\sigma \in M_n$ is a positive semi-definite operator, we define their relative entropy
$$
D( \rho || \sigma ) = \begin{cases}
\tr[ \rho(\log \rho - \log \sigma) ] &\text{if } \supp(\rho) \subseteq \supp(\sigma),\\
+\infty &\text{otherwise},
\end{cases}
$$
where the support of an operator $A \in M_n$ is defined as the orthogonal complement of its kernel, $\supp(A) = \ker(A)^\perp$. 
\end{definition}

The reader is referred to \cite{Wilde} for detailed information regarding quantum relative entropy, we will just record in the next proposition the properties we will need the most (see \cite[Prop. 11.8.2, Ex. 11.8.9, and Cor. 11.9.2]{Wilde}).

\begin{proposition}\label{prop-properties-quantum-relative-entropy}
\begin{enumerate}[(a)]
\item Monotonicity: If $\rho \in M_n$ is a state, and $0 \le \sigma \le \sigma'$ in $M_n$, then $D( \rho || \sigma') \le D( \rho || \sigma )$.
\item Multiplication by scalars: If $\rho \in M_n$ is a state, $\sigma\in M_n^+$, and $c>0$, then $D(\rho||c\sigma) =  D(\rho||\sigma) - \log c$.
\item Joint convexity: If $\lambda_1, \dotsc , \lambda_m \ge 0$ are real numbers with $\sum_{j=1}^m \lambda_j = 1$ then
\[
\sum_{j=1}^m \lambda_j D(\rho_j|| \sigma_j) \ge D\Bigg( \sum_{j=1}^m \lambda_j \rho_j  \Bigg|\Bigg| \sum_{j=1}^m \lambda_j \sigma_j \Bigg)
\]
for any states $\rho_1,\dotsc,\rho_m$ and any $\sigma_1,\dotsc,\sigma_m \in M_n^+$.
\end{enumerate}
\end{proposition}

In analogy with \cite{Odell-Schlumprecht}, we define a second noncommutative generalized Mazur map by minimizing the quantum relative entropy.

\begin{definition}
If $X=(M_n, \n{\cdot})$ is a unitarily invariant strictly convex norm and  and $\rho \in M_n$ is a state, we define $F_X(\rho)$ to be the $\sigma \in \Ball(X)^+$ minimizing the relative entropy with respect to $\rho$, that is,
$$
F_{X}(\rho) = \argmin_{\sigma \in \Ball(X)^+} D( \rho || \sigma ) = \argmin_{\sigma \in \Sph(X)^+} D( \rho || \sigma ).
$$
\end{definition}

Notice that minimizing over $\Ball(X)^+$ or over $\Sph(X)^+$ are equivalent because $D(\rho|| c\sigma) = D(\rho||\sigma) - \log c$ for $c>0$, so in particular the relative entropy decreases when $c>1$.
A few other comments are in order about this definition. First notice that for any state $\rho$ there exists $\sigma \in \Sph(X)^+$ such that $D(\rho||\sigma) < \infty$ (for example, $\sigma = \rho/\n{\rho}_X$).
Therefore, for a fixed $\rho$ the set $\big\{ D(\rho||\sigma) \; : \; \sigma \in \Ball(X)^+  \big\}$ will have a finite infimum. Furthermore, this infimum will be achieved by compactness so it is in fact a minimum.

For the uniqueness, let $\sigma_1, \sigma_2 \in \Sph(X)^+$ both minimize $D(\rho|| \cdot )$ over $\Ball(X)^+$.
From the joint convexity of quantum relative entropy
\[
D\big(\rho|| \tfrac{1}{2}(\sigma_1+\sigma_2)\big)
\le \tfrac{1}{2} \big[ D(\rho|| \sigma_1 ) + D(\rho|| \sigma_2 ) \big] = D(\rho|| \sigma_1 ) \le D\big(\rho|| \tfrac{1}{2}(\sigma_1+\sigma_2)\big),
\]
implying that $\tfrac{1}{2}(\sigma_1+\sigma_2)$ is also a minimizer over $\Ball(X)^+$ and therefore is also on the unit sphere, so by strict convexity $\sigma_1=\sigma_2$.

The next lemma is a technical calculation that will be needed in the subsequent result.

\begin{lemma}\label{lemma-logs-are-close}
Let $A, B \in M_n^+$ and $\varepsilon\in(0,1)$.
Then
$$
\n{\log(A+\varepsilon B) - \log(B+\varepsilon A)}_\infty \le |\log \varepsilon|.
$$
\end{lemma}

\begin{proof}
Note that $$
\varepsilon(B + \varepsilon A) \le A + \varepsilon B \le \frac{1}{\varepsilon}(B + \varepsilon A),
$$
from where by the operator monotonicity of the logarithm \cite[Thm. 2.6]{Carlen}
$$
\log(\varepsilon(B + \varepsilon A)) \le \log( A + \varepsilon B) \le \log( \frac{1}{\varepsilon}(B + \varepsilon A)),
$$
thus
$$
(\log \varepsilon) P \le \log( A + \varepsilon B) - \log(B + \varepsilon A) \le (\log \frac{1}{\varepsilon}) P
$$
where $P$ is the orthogonal projection onto the support of $B + \varepsilon A$.
The conclusion follows.
\end{proof}

The following is now a quantum version of \cite[Lemma 2.3.(B)]{Odell-Schlumprecht}.

\begin{lemma}\label{lemma-uniform-continuity}
Let $X=(M_n, \n{\cdot})$ be a unitarily invariant
strictly convex
norm on $M_n$.
For any states $\rho_1, \rho_2$ in $M_n$ with $\n{\rho_1 - \rho_2}_1 \le 1$ we have that 
$$
\n{\tfrac{1}{2}\big( F_{X}(\rho_1) + F_{X}(\rho_2) \big)}_X \ge 1 - \sqrt{\n{\rho_1 - \rho_2}_1}.
$$
\end{lemma}

\begin{proof}
Let $A_1 = F_{X}(\rho_1)$, $A_2 = F_{X}(\rho_2)$, and let $\varepsilon$ be defined by $\n{\tfrac{1}{2}\big( A_1 + A_2 \big)} = 1 - 2\varepsilon$.
If $A_1=A_2$ there is nothing to prove since the left-hand side of the desired inequality is then equal to 1, so we may assume $A_1\not=A_2$ which implies $0<\varepsilon<1$.
Set
$$
\widetilde{A}_1 = A_1 + \varepsilon A_2, \qquad \widetilde{A}_2 = A_2 + \varepsilon A_1.
$$
Note that $\big\| \tfrac{1}{2}\big( \widetilde{A}_1 + \widetilde{A}_2 \big) \big\| \le 1 - \varepsilon$, and by Lemma \ref{lemma-logs-are-close} we have
$\big\| \log\widetilde{A}_1-\log\widetilde{A}_2 \big\|_\infty \le |\log \varepsilon|$.
Moreover, note that $\ker \widetilde{A}_j = (\ker A_1) \cap (\ker A_2)$, and therefore $\supp(A_i) \subseteq \supp(\widetilde{A}_j)$.
This implies that the quantities $D(\rho_i||\widetilde{A}_j)$ are all finite.

By monotonicity of quantum relative entropy,
together with the optimality in the definition of $F_X(\rho_1)$, and
using the behavior of quantum relative entropy under multiplication by scalars,
\begin{multline*}
D( \rho_1 || \widetilde{A}_1 ) \le D( \rho_1 || A_1 ) \le D\Big( \rho_1 || \tfrac{1}{2(1-\varepsilon)}( \widetilde{A}_1+ \widetilde{A}_2 ) \Big)\\
=  D\Big( \rho_1 || \tfrac{1}{2}( \widetilde{A}_1+ \widetilde{A}_2 ) \Big) + \log(1-\varepsilon).
\end{multline*}
which implies, now using convexity of the quantum relative entropy,
\begin{multline*}
\varepsilon \le |\log(1-\varepsilon)| \le D\Big( \rho_1 || \tfrac{1}{2}( \widetilde{A}_1+ \widetilde{A}_2 ) \Big) - D( \rho_1 || \widetilde{A}_1 ) 
\\
\le \tfrac{1}{2} D(\rho_1||\widetilde{A}_1) + \tfrac{1}{2} D(\rho_1||\widetilde{A}_2) - D( \rho_1 || \widetilde{A}_1 )
= -\tfrac{1}{2} D(\rho_1||\widetilde{A}_1) + \tfrac{1}{2} D(\rho_1||\widetilde{A}_2).
\end{multline*}
Similarly,
$$
\varepsilon \le -\tfrac{1}{2} D(\rho_2||\widetilde{A}_2) + \tfrac{1}{2} D(\rho_2||\widetilde{A}_1).
$$
Averaging both inequalities,
\begin{multline*}
\varepsilon \le \tfrac{1}{4}\big[ - D(\rho_1||\widetilde{A}_1) +  D(\rho_1||\widetilde{A}_2) - D(\rho_2||\widetilde{A}_2) + D(\rho_2||\widetilde{A}_1) \big] \\
= \tr\big[ (\rho_1-\rho_2)\big( \log\widetilde{A}_1-\log\widetilde{A}_2 \big)  ]
\le
\n{\rho_1-\rho_2}_1 \n{ \log\widetilde{A}_1-\log\widetilde{A}_2 }_\infty \\
\le |\log \varepsilon| \n{\rho_1-\rho_2}_1 \le \frac{1}{\varepsilon} \n{\rho_1-\rho_2}_1,
\end{multline*}
from where the desired result follows.
\end{proof}

The following is a first step towards a noncommutative version of \cite[Prop. 2.6]{Odell-Schlumprecht}.

\begin{proposition}\label{prop-FX-unif-coninuous on positive}
Let $X=(M_n, \n{\cdot})$ be an unitarily invariant matrix norm which is uniformly convex.
Then $F_X : \Sph(S_1^n)^+ \to \Sph(X)^+$ is uniformly continuous; moreover, the modulus of continuity of $F_X$ depends only on the modulus of uniform convexity of $X$.
\end{proposition}

\begin{proof}
Note that from Lemma \ref{lemma-uniform-continuity} we get that for any  $\rho_1, \rho_2 \in \Sph(S_1^n)^+$ with $\n{\rho_1 - \rho_2}_1 \le 1$ we have
$$
\sqrt{\n{\rho_1 - \rho_2}_1} \ge 1 - \n{\tfrac{1}{2}\big( F_{X}(\rho_1) + F_{X}(\rho_2) \big)}_X \ge \delta_X\big( \n{F_{X}(\rho_1) - F_{X}(\rho_2)}_X\big).
$$
Therefore, whenever $\n{F_{X}(\rho_1) - F_{X}(\rho_2)}_X \ge \varepsilon$, we must have $\n{\rho_1 - \rho_2}_1 \ge \big( \delta_X(\varepsilon)\big)^2$.
Equivalently, there exists $\eta(\varepsilon) = \big( \delta_X(\varepsilon)\big)^2$ such that
$\n{\rho_1 - \rho_2}_1< \eta(\varepsilon)$ implies $\n{F_{X}(\rho_1) - F_{X}(\rho_2)}_X < \varepsilon$.
Letting $\eta(0)=0$, the function $\eta$ is continuous and strictly increasing on $[0,2]$.
Therefore $\eta$ has an inverse $g$, depending only on the modulus of convexity of $X$, such that $$
\n{F_{X}(\rho_1) - F_{X}(\rho_2)}_X \le g\big( \n{\rho_1 - \rho_2}_1 \big),
$$
which proves the desired result.
\end{proof}

In order to study the uniform continuity of $F_X^{-1}$, it will be convenient to develop an alternative description of it in terms of the duality map.
Note that such an alternative description is well-known in the classical case, in fact \cite[Chap. 9]{Benyamini-Lindenstrauss} uses only the duality map description instead of looking at entropy minimization.

\begin{proposition}\label{prop-G-unif-coninuous}
Let $X=(M_n, \n{\cdot})$ be an unitarily invariant matrix norm which is uniformly smooth.
Define $G : X \to S_1^n$ by $G(A) = |J_X(A)|A$.
Then $G$ is uniformly continuous on $\Sph(X)$, and its modulus of uniform continuity depends only on the modulus of uniform smoothness of $X$.
Moreover, when $A \in \Sph(X)^+$ we have $G(A) \in \Sph(S_1^n)^+$.
\end{proposition}

\begin{proof}
For simplicity we write $J$ instead of $J_X$.
This will be an easy triangle inequality argument analogous to the one in \cite{Odell-Schlumprecht}: for $A,B \in \Sph(X)$,
\begin{multline*}
\n{G(A)-G(B)}_1 \le \n{ |J(A)|\big(A-B\big) }_1 + \n{\big(|J(A)|-|J(B)|\big)B}_1 \\
\le \n{|J(A)|}_{X^*} \n{A-B}_X + \n{|J(A)|-|J(B)|}_{X^*}\n{B}_X   \\
\le \n{A-B}_X + \n{|J(A)|-|J(B)|}_{X^*}.
\end{multline*}
Since $X$ is uniformly smooth, $J$ is uniformly continuous on $\Sph(X)$ and its modulus of uniform continuity depends only on the modulus of uniform smoothness of $X$ \cite[Prop. A.5]{Benyamini-Lindenstrauss}.
Since $X$ is unitarily invariant so is $X^*$, and therefore the absolute value is uniformly continuous on $\Sph(X^*)$ with a universal modulus of uniform continuity \cite[Thm. X.2.1]{Bhatia}.
It is now clear that $G$ is uniformly continuous, and its modulus of uniform continuity depends only on the modulus of uniform smoothness of $X$.
The fact that when $A \in \Sph(X)^+$ we have $G(A) \in \Sph(S_1^n)^+$ follows from Lemma \ref{lemma-G-gives-states}.
\end{proof}

Next we prove that $G$ is in fact the inverse of $F_X$. Our approach is inspired by standard techniques in convexity and quantum entropy as in \cite{Carlen}.

\begin{lemma}\label{lemma-G-is-F-inverse}
Let $X=(M_n, \n{\cdot})$ be an unitarily invariant matrix norm which is uniformly convex and uniformly smooth.
Then for any state $\rho \in \Sph(S_1^n)^+$ we have $G(F_X(\rho))  = \rho$, and for any matrix $A \in \Sph(X)^+$ we have $F_X(G(A)) = A$.
\end{lemma}

\begin{proof}
  By density, we can assume that $\rho$ has full support.
 Let $A= F_X(\rho)$, that is, $A$ minimizes $D(\rho|| \cdot )$ over $\Ball(X)^+$; as pointed out above, in fact $A \in \Sph(X)^+$.
Notice that since $\rho$ has full support, so must $A$ (i.e. $A$ is invertible).
 
Consider the set 
$$
\mathcal{F} = \{ C \in M_n^+ \; : \; D(\rho||C) < D(\rho||A) \}.
$$
By the joint convexity of quantum relative entropy,
the set $\mathcal{F}$ is convex.
Since $\mathcal{F}$ is disjoint from the set $\Ball(X)^+$, they can be separated by a hyperplane $H$.
But $A$ belongs to both the boundary of $\Ball(X)^+$ and the closure of $\mathcal{F}$, so the hyperplane must be a supporting hyperplane for $\Ball(X)^+$ at $A$ and thus of the form $\{ C \in M_n \; : \; \tr(B C) = 1\}$ for $B=J(A)$.
Note that in particular we have $\tr(BA)=1$, $\tr(BA') \le 1$ for every $A'\in \Ball(X)^+$, and $B \ge 0$. 

Let $D \in M_n$ be a Hermitian matrix such that $\tr(BD)=0$.
For values of $t$ small enough,
observe that the matrix $A + tD$ is still positive definite and moreover it belongs to the hyperplane $H$, and therefore $D(\rho||A+tD) \ge D(\rho||A)$.
It follows that the function $f(t) = -\tr( \rho \log(A+tD) )$ has a minimum at $t=0$, and therefore its derivative at $t=0$ is equal to $0$.
Spectral calculus together with a power series representation for the logarithm then yields
$$
0 = f'(0) = \tr( \rho A^{-1} D ).
$$  
But $D$ was an arbitrary Hermitian matrix satisfying $\tr(BD)=0$, which implies $B = \rho A^{-1}$ and therefore $BA=\rho$, that is, $G(F_X(\rho))=\rho$.  

If now $A \in \Sph(X)^+$ is arbitrary, since $G(A)$ is a state by Proposition \ref{prop-G-unif-coninuous}, it follows from the first part that $G(F_X(G(A))) = G(A)$,
and now from Lemma \ref{lemma-unique-factorization} we conclude
that $F_X(G(A)) =A$.
\end{proof}

Now, let us extend the definition of $F_X$ to all of $\Sph(S_1^n)$:
for $A \in \Sph(S_1^n)$ with polar decomposition $A=U|A|$, we define $F_X(A) = U F_X(|A|)$.
Given $A \in \Sph(X)$ with polar decomposition $A=U|A|$ note that $J(A)=J(|A|)U^*$,
since $\n{ J(|A|)U^* }_{X^*}=1$ and
$$
\tr\big[ J(|A|)U^* A \big] = \tr\big[ J(|A|)|A| \big] = 1.
$$
It then follows that
$$
|J(A)| = \big( J(A)^*J(A) \big)^{1/2} = \big( U J(|A|)J(|A|)U^* \big)^{1/2} = UJ(|A|)U^*
$$
and therefore
$$
G(A) = |J(A)|A = UJ(|A|)U^*U|A| = UJ(|A|)|A| = U G(|A|).
$$
From this last expression, it is easy to check that for any $A \in \Sph(X)$ with polar decomposition $A=U|A|$ we have
$$
F_X(G(A)) = F_X( U G(|A|) ) = U F_X( G(|A|) ) = U|A| = A,
$$
whereas for $A \in \Sph(S_1^n)$ with polar decomposition $A=U|A|$ we have
$$
G(F_X(A)) = G( U F_X(|A|) ) = U G(F_X(|A|)) = U|A| = A.
$$
showing that $F_X :  \Sph(S_1^n)\to \Sph(X)$ and $G : \Sph(X) \to \Sph(S_1^n)$ are inverses.

\begin{remark}
It follows from the arguments above that when $X$ is a unitarily invariant matrix norm on $M_n$ which is uniformly convex and uniformly smooth, then
$F_X^{-1} : \Sph(X) \to \Sph(S_1^n)$ is uniformly continuous with modulus of continuity depending only on the modulus of uniform smoothness of $X$.
However, for the entropy minimization generalized Mazur map $F_X$ itself, we have been unable to bridge the gap between it being uniformly continuous on the positive part of the sphere $\Sph(S_1^n)^+$ (Proposition \ref{prop-FX-unif-coninuous on positive}), and it being uniformly continuous on the entire sphere $\Sph(S_1^n)$.
Note that there is no obvious way to apply the  $2\times 2$ tricks which featured heavily in Section \ref{sec-convexification}: while we could explicitly calculate how the map $G_p$ acts on certain block diagonal and block antidiagonal matrices, it is not clear how the map $F_X$ acts on such matrices.

Note that if the aforementioned gap could be bridged, a proof of Theorem \ref{thm-main} in the style of Odell and Schlumprecht would immediately follow. While the results in this section deal only with the finite-dimensional situation, those finite-dimensional pieces could be glued together in the same way as in \cite{Odell-Schlumprecht} or \cite[Cor. 9.6]{Benyamini-Lindenstrauss}.
\end{remark}

\section{The main result}\label{sec-main}

While we have not been able to obtain a full proof of a noncommutative version of the Odell--Schlumprecht theorem by following the entropy minimization route, the complex interpolation approach of Daher \cite{Daher} will provide the missing link. The key result is the following theorem from \cite{Daher} (also independently obtained by Kalton).

\begin{theorem}\label{thm-Daher}
Let $X_0,X_1$ be an interpolation pair with one of them being uniformly convex. Then for any $\theta, \eta \in (0,1)$, $\Sph(X_\theta)$ and $\Sph(X_\eta)$ are uniformly homeomorphic.
\end{theorem}

\begin{remark}
Since the $p$-convexification $E^{(p)}$ can be realized as an interpolation space between $E$ and $\ell_\infty$ \cite{JLS}, Theorem \ref{thm-Daher} and Proposition \ref{prop:interpolation-unitarily-invariant-ideals}
together yield uniform homeomorphisms between the spheres of $E^{(p)}$ and $E^{(q)}$ for any $p,q\in(1,\infty)$ whenever 
$E$ is uniformly convex. Compare with Theorem \ref{thm-homeomorphism-convexification}.
\end{remark}

The following is a noncommutative version of \cite[Cor. 1]{Daher}. For the definitions of $p$-convexity and $q$-concavity, and the associated constants, see \cite[Sec. 1.d]{Lindenstrauss-Tzafriri-II}.

\begin{proposition}\label{prop-pconvex-and-qconcave-give-homeomorphism}
Let $1 < p \le 2 \le q <\infty$. Let $E$ be a $1$-symmetric sequence space which is $p$-convex and $q$-concave.
Then $\Sph(S_E)$ is uniformly homeomorphic to $\Sph(S_2)$.
\end{proposition}

\begin{proof}
After renorming, we may assume that $M^{(p)}(E)=M_{(q)}(E) = 1$ \cite[Prop. 1.d.8]{LT-II}.
While this is not part of the statement, following the proof of \cite[Prop. 1.d.8]{LT-II} we note that since $E$ is $1$-symmetric then so is the renorming (it is obtained through a combination of convexifications, concavifications and passing to duals, all of which will preserve the property of being $1$-symmetric).
Note that this will change the unit sphere of $S_E$, but the old and new spheres will be clearly bi-Lipschitz equivalent so this will not affect our conclusion. It now follows from a theorem of Pisier \cite{Pisier-interpolation-lattices} that there exist $\theta\in(0,1)$ and a $1$-unconditional sequence space $Y$ with $\ell_1 \subseteq Y \subseteq \ell_\infty$ such that $E = (Y,\ell_2)_\theta$.
Let $\eta \in(0, \theta)$.
Since $M^{(2)}(\ell_2) = M^{(1)}(Y) =M_{(2)}(\ell_2) = M_{(\infty)}(Y)=1$, it follows that $Z = (Y,\ell_2)_\eta$ is $\alpha$-convex and $\beta$-concave with constants equal to 1, where $1/\alpha = 1-\eta+\eta/2$ and $1/\beta = \eta/2$ (see \cite[Remark 2.1]{Pisier-interpolation-lattices}).
By reiteration, we can write $E = (Z,\ell_2)_{\theta'}$ for some other $\theta'\in(0,1)$.
Since $(Z,Z^*)_{1/2} = \ell_2$ \cite{Calderon}, once again by reiteration we get $E = (Z,Z^*)_{\theta'/2}$.
It now follows from Proposition \ref{prop:interpolation-unitarily-invariant-ideals} that $S_{E} = (S_Z,S_{Z^*})_{\theta/2}$, and also $S_2 = (S_Z,S_{Z^*})_{1/2}$.
From \cite[Thm. 2]{TJ-uc}, since $Z$ is $\alpha$-convex and $\beta$-concave with constants equal to 1, we get that $S_Z$ is uniformly convex. It then follows from Theorem \ref{thm-Daher} that $\Sph(S_{E})$ and $\Sph(S_2)$ are uniformly homeomorphic. 
\end{proof}

We can now prove the desired result.

\begin{theorem}\label{thm-main}
Let $E$ be a 1-symmetric sequence space.
The following are equivalent:
\begin{enumerate}[(a)]
\item $\Sph(S_E)$ is uniformly homeomorphic to $\Sph(S_1)$.
\item $S_E$ does not contain $\ell_\infty^n$'s uniformly.
\item $E$ does not contain $\ell_\infty^n$'s uniformly.
\end{enumerate}
\end{theorem}

\begin{proof}
The equivalence between (b) and (c) is well-known \cite{Garling-TJ}, as not containing $\ell_\infty^n$'s uniformly is equivalent to having finite cotype by the Maurey--Pisier theorem .

Suppose that (a) holds. 
It is also well-known that $\Sph(S_1)$ is uniformly homeomorphic to $\Sph(S_2)$ (see \cite[Thm. 4.1]{Ozawa} or \cite{Ricard} for proofs), so the same argument as in \cite{Odell-Schlumprecht} applies: it follows from a result of Enflo \cite{Enflo} that in this case $S_E$  cannot contain $\ell_\infty^n$'s uniformly, yielding (b).

Now suppose that (c) holds.
Our argument follows \cite{Daher}. 
Once again appealing to the Maurey--Pisier theorem, $E$ has cotype $q'$ for some $q'<\infty$,
and therefore $E$ is $q$-concave for any $q>q'$ \cite[Cor. 1.f.9]{LT-II}.
The $3$-convexification $E^{(3)}$ of $E$ will then be $3$-convex and $3q$-concave.
By Proposition \ref{prop-pconvex-and-qconcave-give-homeomorphism} $\Sph(S_{E^{(3)}})$ is uniformly homeomorphic to $\Sph(S_2)$, and by Theorem \ref{thm-homeomorphism-convexification} $\Sph(S_E)$ is uniformly homeomorphic to $\Sph(S_{E^{(3)}})$.
\end{proof}

\begin{remark}
As far as we know
Theorem \ref{thm-main} provides new examples of spaces whose unit sphere is uniformly homeomorphic to the unit sphere  of a Hilbert space beyond the ones that were already known, e.g. the ones mentioned in \cite[Sec. 3.2]{Mimura}  or
\cite[Chap. 9]{Benyamini-Lindenstrauss}.
Also, it increases the family of known spaces for which the finite cotype condition is enough (recall that for general spaces, finite cotype is not enough \cite[Example 9.23]{Benyamini-Lindenstrauss}).
\end{remark}

A refinement of the condition of having unit sphere uniformly homeomeomorphic to that of a Hilbert space, known as Property (H), is noteworthy in the $K$-Theory of Operator Algebras. Before stating the definition, recall that a \emph{paving} of a separable Banach space is an increasing sequence of finite-dimensional subspaces whose union is dense in the whole space.

\begin{definition}
A Banach space $X$ has \emph{Property (H)} if there exist pavings $(X_n)_{n=1}^\infty$ of $X$ and $(Y_n)_{n=1}^\infty$ of a separable Hilbert space $H$, and a uniform homeomorphism $f : \Sph(X) \to \Sph(H)$, such that for each $n\in\N$ the restriction of $f$ to $\Sph(X_n)$ is a (uniform) homeomorphism onto $\Sph(Y_n)$. 
\end{definition}

Kasparov and Yu
\cite{Kasparov-Yu} proved the strong Novikov conjecture for groups coarsely embeddable into Banach spaces satisfying this property. 
See also \cite{MR3325537} for a  geometric analogue of this result: any discrete metric space with bounded geometry which coarsely embeds into a Banach space with Property (H) satisfies the coarse Novikov conjecture.

By \cite[Thm. 6.2]{Cheng-Wang}, the sequence spaces covered in the classical Odell--Schlumprecht theorem have Property (H).
We now show that the same conclusion holds for our noncommutative version.

\begin{theorem}
Let $E$ be a 1-symmetric sequence space which does not contain $\ell_\infty^n$'s uniformly. Then $S_E$ has Property (H).
\end{theorem}

\begin{proof}
The proof of Theorem \ref{thm-main} provides a uniform homeomorphism $f : \Sph(S_E) \to \Sph(S_2)$.
We just need to check that the Property (H) condition is satisfied for $f$ with respect to the canonical pavings $(S_E^n)_{n=1}^\infty$ and $(S_2^n)_{n=1}^\infty$.

The uniform homeomorphism $f$ was constructed in three steps, so we will track what happens to $\Sph(S_E^n)$ at each step. The first step is given by the map $G_3^{-1}$, which maps $\Sph(S_E^n)$ onto $\Sph(S^n_{E^{(3)}})$.
The second step (contained in the proof of Proposition \ref{prop-pconvex-and-qconcave-give-homeomorphism}) is induced by a renorming of $E^{(3)}$ which yields another $1$-symmetric sequence space $F$. Clearly this second step maps $\Sph(S^n_{E^{(3)}})$ onto $\Sph(S^n_F)$.
The third step is given by the complex interpolation argument appearing in the proof of Proposition \ref{prop-pconvex-and-qconcave-give-homeomorphism}, which in turn depends on Theorem \ref{thm-Daher}.
A careful look at the proof of the latter theorem, as presented in \cite[Thm. 9.12]{Benyamini-Lindenstrauss} shows that the third step maps $\Sph(S^n_F)$ onto $\Sph(S^n_2)$, finishing the proof.
\end{proof}

\bibliography{references}
\bibliographystyle{amsalpha}

\end{document}